\documentclass[12pt]{amsart}

\usepackage{latexsym, amssymb, amsmath}

\usepackage{amsfonts, graphicx}
\usepackage{graphicx,color}
\newcommand{\be}{\begin{equation}}
\newcommand{\ee}{\end{equation}}
\newcommand{\beq}{\begin{eqnarray}}
\newcommand{\eeq}{\end{eqnarray}}

\newtheorem{thm}{Theorem}[section]

\newtheorem{cor}{Corollary}[section]

\theoremstyle{remark}

\numberwithin{equation}{section}

\numberwithin{equation}{section}
\newtheorem{theorem}{Theorem}[section]
\newtheorem{remark}[theorem]{Remark}
\newtheorem*{claim*}{Claim}
\DeclareMathOperator{\Rm}{Rm}
\DeclareMathOperator{\Ric}{Ric}
\DeclareMathOperator{\Rc}{Rc}
\DeclareMathOperator{\Inj}{inj}
\DeclareMathOperator{\Vol}{Vol}
\DeclareMathOperator{\vol}{vol}

\def\K{K\"ahler }

\begin{document}

\title{A note on a diffeomorphism criterion via long-time Ricci flow}
\author[S. Huang]{Shaochuang Huang}
\address[Shaochuang Huang]{College of Science, Sun Yat-sen University, Shenzhen, Guangdong, China.}
\email{huangshch23@mail.sysu.edu.cn}
\thanks{S. Huang is partially supported by a regular fund from Shenzhen Science and Technology Innovation Bureau No.JCYJ20240813151005007 and a start-up fund from SYSU}

\author[Z. Peng]{Zhuo Peng}
\address[Zhuo Peng]{College of Science, Sun Yat-sen University, Shenzhen, Guangdong, China.}
\email{pengzh55@mail2.sysu.edu.cn}

\subjclass[2020]{Primary 53E20}
\date{{\today}}

\begin{abstract} 
In this note, we give a  diffeomorphism (to $\mathbb{R}^n$) criterion via long-time Ricci flow and show some applications. In particular, we provide an affirmative answer that  the conclusion in \cite{CHL24} and \cite{Mar25} about manifolds with small curvature concentration can be improved to diffeomorphism in dimension $4$.   
\end{abstract}

\maketitle

\section{Introduction}

Ricci flow, which was introduced by R. Hamilton in 1982\cite{Ham82}, is a powerful tool to study the topology and geometry of the underlying manifold, see \cite{Per02},\cite{BW08},\cite{BS09} for example. In this note, we will focus on studying the topology of complete non-compact manifolds via Ricci flow.   For the compact case, the model is the standard sphere $\mathbb{S}^n$ or the complex projective space $\mathbb{CP}^n$; while for the non-compact case, the model is the standard Euclidean space $\mathbb{R}^n$ or $\mathbb{C}^n$. In \cite{Shi90Rd}, W.-X. Shi initiates a programme to study uniformization conjecture on \K geometry via Ricci flow  under the supervision by S.-T. Yau. The uniformization conjecture states that any $n$-dimensional complete  non-compact \K manifold with positive holomorphic bisectional curvature must be bi-holomorphic to the standard complex Euclidean space $\mathbb{C}^n$. Shi aims to construct a complete long-time solution to the Ricci flow starting from such \K manifold in order to study the topology of this underlying manifold. This  programme is now  achieved by many works under the maximal volume growth assumption, see \cite{CTZ04},\cite{CT06},\cite{HT18},\cite{LT20} for example. 

On the other hand, it is also natural to study the topology of a complete non-compact  Riemannian manifold via Ricci flow. In fact, Chen, Tang and Zhu first showed that the topology of certain \K manifolds with positive holomorphic bisectional curvature must be $\mathbb{R}^{2n}$ via Ricci flow in \cite{CZ03},\cite{CTZ04}, see also \cite{NT04}.  In \cite{Wang20}, B. Wang investigates the local entropy and pseudo-locality theorems for the Ricci flow. Then he applies these to provide a Ricci flow proof for an important result, which was first proved by J. Cheeger and T. Colding \cite{CC97} via their theory on studying Gromov-Hausdorff convergence for Riemannian manifolds with Ricci curvature bounded from below, stating that an $n$-dimensional complete non-compact Riemannian manifold with non-negative Ricci curvature and almost Euclidean volume ratio must be diffeomorphic to $\mathbb{R}^n$. In \cite{HL21}, F. He and M.-C. Lee show that an $n$-dimensional complete non-compact Riemannian manifold with weakly $PIC_1$ and  maximal volume growth must be diffeomorphic to $\mathbb{R}^n$ via Ricci flow. In \cite{CHL24}, Chan, Lee and the first named author show that  an $n$-dimensional complete non-compact Riemannian manifold with small curvature concentration under certain assumptions must be homeomorphic to $\mathbb{R}^n$ if $n\geq 4$ and diffeomorphic to  $\mathbb{R}^n$ if $n\geq 5$. Later, A. Martens \cite{Mar25} removes the scalar curvature assumption in \cite{CHL24}.

The initial motivation of this note is to investigate whether in dimension $4$, the conclusion in \cite{CHL24} and \cite{Mar25} can be improved to diffeomorphism.  We will provide an  affirmative answer by showing the following diffeomorphism (to $\mathbb{R}^n$) criterion via long-time Ricci flow. 

\begin{thm} \label{mainthm1}
    Suppose $(M^n,g(t))$ is a smooth complete long-time solution to the Ricci flow on $M^n\times[0,+\infty)$ satisfying 
    \begin{enumerate}
        \item[\textup{(\romannumeral1)}] $\displaystyle \Ric(g(t))\ge-\frac{\psi}{t}\text{ for some constant }0<\psi<\frac{1}{2}$;
        \item[\textup{(\romannumeral2)}] $\displaystyle\Inj(g(t))\ge\beta\sqrt{t}\text{ for some constant }\beta>0$.
    \end{enumerate}
    Then $M^n$ is diffeomorphic to $\mathbb{R}^n$.
\end{thm}

\begin{remark} \label{rmkkappa}
    We can see from the proof of Method 1 of Theorem \ref{mainthm1}  that if the injectivity radius condition \textup{(\romannumeral2)} is  changed to $\Inj(g(t))\ge\beta t^{\kappa}\text{ for some } \kappa>0$, then we only need to require $0<\psi<\kappa$ in the curvature condition \textup{(\romannumeral1)}.
\end{remark}

As applications, we first obtain the following, which confirms that  the conclusion in \cite{CHL24} and \cite{Mar25} can be improved to diffeomorphism in dimension $4$. 

\begin{cor}\label{small curvature concentration}
    For $n\geq4$ and $A,v_0>0$, there is $\sigma>0$ depending only on $n,A,v_0$ such that if $(M^n,g_0)$ is an $n$-dimensional complete non-compact manifold satisfying
    \begin{enumerate}
        \item[\textup{(\romannumeral1)}] $\Vol_{g_0}(B_{g_0}(x,r))\le v_0r^n$ for all $r>0$;
        \item[\textup{(\romannumeral2)}] $\bar{\nu}(M,g_0)\ge-A$;
        \item[\textup{(\romannumeral3)}] $\displaystyle\bigg(\int_M|\Rm(g_0)|^{n/2}d\vol_{g_0}\bigg)^{2/n}\le\sigma$;
        \item[\textup{(\romannumeral4)}] $\Ric(g_0)\ge-k$ for some $k\in\mathbb{R}$,
    \end{enumerate}
    for all $x\in M$. Then $M^n$ is diffeomorphic to $\mathbb{R}^n$.
\end{cor}

Another interesting application is the following, which removes the almost Euclidean assumption in the main result of \cite{CC97} and \cite{Wang20} mentioned above in dimension $3$.

\begin{cor}\label{3d}
    Suppose $(M^3,g_0)$ is a $3$-dimensional complete non-compact manifold satisfying
    \begin{enumerate}
        \item[\textup{(\romannumeral1)}] $\Ric(g_0)\ge0$;
        \item[\textup{(\romannumeral2)}] $g_0$ has maximal volume growth.
    \end{enumerate}
    Then $M^3$ is diffeomorphic to $\mathbb{R}^3$.
\end{cor}

\begin{remark} This result should be known by experts. Indeed, G.Liu has classified the topology of all $3$-dimensional complete non-compact manifold with non-negative Ricci curvature in \cite{Liu13} via minimal surface theory; if one further assume the manifold has maximal volume growth, then a simple argument can rule out other cases and the topology of  the manifold is the trivial one. We provide a Ricci flow proof  here by combining Theorem \ref{mainthm1} and a short-time existence result of Ricci flow for this setting in \cite{ST21}. Note that this result is not true in higher dimension, since Eguchi-Hanson metric is a counter-example, see \cite{ST22} for some related discussion. In higher dimension, one should assume stronger curvature condition or almost Euclidean condition.
\end{remark}

The next application studies the topology of complete non-compact manifold with small curvature concentration related to $W^{1,2}$-Sobolev inequality, which improves the main result in \cite{CM24} where some $L^p$-bound for Riemann curvature tensor is needed. 

\begin{cor}\label{bounded curvature}
    For $n\ge3$, there is a dimensional constants $\sigma_n>0$ such that if $(M^n,g_0)$ is an $n$-dimensional complete non-compact  manifold with \textbf{bounded curvature} satisfying
    \begin{enumerate}
        \item[\textup{(\romannumeral1)}] $\displaystyle\bigg(\int_M u^\frac{2n}{n-2}d\vol_{g_0}\bigg)^\frac{n-2}{n}\le C_S\int_M|\nabla u|^2d\vol_{g_0}$ for all $u\in W^{1,2}(M, g_0)$;
        \item[\textup{(\romannumeral2)}] $\displaystyle\bigg(\int_M|\Rm(g_0)|^{n/2}d\vol_{g_0}\bigg)^{2/n}\le\sigma_n\frac{1}{C_{S}}$,
    \end{enumerate}
    where $C_{S}$ is a positive constant. Then $M^n$ is diffeomorphic to $\mathbb{R}^n$.
\end{cor}

Note that the condition (ii) in Corollary \ref{bounded curvature} is explicit about the dependence of the Sobolev constant $C_{S}$, while Corollary \ref{small curvature concentration} merely implies the case of an implicit dependence of the Sobolev constant $C_{S}$ in condition (ii). We wonder that whether one can remove the bounded curvature assumption in Corollary \ref{bounded curvature}, see \cite{CWY22} for the compact case. 

The paper is organized as follows. In Section 2, we will prove the diffeomorphism (to $\mathbb{R}^n$) criterion via long-time Ricci flow, namely Theorem \ref{mainthm1}. In Section 3, we will show the applications, namely Corollary \ref{small curvature concentration}, \ref{3d}, \ref{bounded curvature}. 

{\it Acknowledgement:}
The authors would like to thank Liang Cheng and Xian-Tao Huang for some useful discussion.

\section{Diffeomorphism  criterion via long-time Ricci flow}

In this section, we will provide two methods to prove the diffeomorphism (to $\mathbb{R}^n$) criterion via long-time Ricci flow, namely Theorem \ref{mainthm1}.

\begin{proof}[Proof of Theorem \ref{mainthm1} (Method 1)]

    This method adapts the idea in the proof of \cite[Theorem 1.1]{HL21}, which is more or less standard in differential topology. 
    
    Fix $x_0\in M$. Let $exp_t:\mathbb{R}^n\to M$ be the exponential map w.r.t $g(t)$ at $x_0$ and denote $\mathcal{E}_t=exp_t^{-1}$. For simplicity, we denote
    \begin{align*}
        \tilde{B}(r)&:=\{x\in\mathbb{R}^n\big||x|<r\}, \\
        B_t(r)&:=B_{g(t)}(x_0,r),
    \end{align*}
    then $exp_t(\tilde{B}(r))=B_t(r)$ as long as $r<inj(g(t))$. 
    \par
    Let $\displaystyle t_i:=i\text{ and }R_i:=\bigg(\frac{1}{2}-\frac{1}{2^{i+1}}\bigg)\beta\sqrt{t_i}$ for $i=1,2,3,\cdots$. Denote $\displaystyle\varepsilon_i:=\frac{1}{2^{i+2}}\beta\sqrt{t_{i+1}}$. We claim that for each $t\in[t_i,t_{i+1}]$,
    \begin{equation}
        B_{t_1}\bigg(\frac{\beta}{4}t_i^{\frac{1}{2}-\psi}\bigg)\subset B_{t_i}(R_i)\subset B_t\bigg(R_i\bigg(\frac{t}{t_i}\bigg)^\frac{1}{2}\bigg)\subset B_{t_{i+1}}(R_{i+1}-\varepsilon_i). \label{ballnested}
    \end{equation}

     To see \eqref{ballnested}, we first note that 
     by the geodesic length evolution equation along the Ricci flow and (\romannumeral1), we have
    \[
        \frac{d}{dt}\log d_{g(t)}(x,x_0)=-\frac{1}{|\gamma|}\int_\gamma\Rc(\dot{\gamma},\dot{\gamma})d\theta\le\frac{\psi}{t}.
    \]
    It follows that
    \begin{equation}
        \frac{d_{g(t^\prime)}(x,x_0)}{d_{g(t)}(x,x_0)}\le\bigg(\frac{t^\prime}{t}\bigg)^\psi \label{distdist}
    \end{equation}
    for any $t^\prime\geq t$.
    
    For all $x\in B_{t_1}\big(\frac{\beta}{4}t_i^{\frac{1}{2}-\psi}\big)$, by (\ref{distdist}),
    \[
        d_{t_i}(x,x_0)\le d_{t_1}(x,x_0)t_i^\psi<\frac{\beta}{4}\sqrt{t_i}\le R_i.
    \]
    For all $x\in B_{t_i}(R_i)$, by (\ref{distdist}),
    \[
        d_t(x,x_0)\le d_{t_i}(x,x_0)\bigg(\frac{t}{t_i}\bigg)^\psi< R_i\bigg(\frac{t}{t_i}\bigg)^\frac{1}{2}.
    \]
    For all $x\in B_t\big(R_i\big(\frac{t}{t_i}\big)^\frac{1}{2}\big)$, by (\ref{distdist}),
    \[
        d_{t_{i+1}}(x,x_0)\le d_t(x,x_0)\bigg(\frac{t_{i+1}}{t}\bigg)^\psi< R_i\bigg(\frac{t_{i+1}}{t_i}\bigg)^\frac{1}{2}=R_{i+1}-\varepsilon_i.
    \]
    By the first inclusion in (\ref{ballnested}), $\{B_{t_i}(R_i)\}$ is an exhaustion of M, and obviously $\{\tilde{B}(R_i)\}$ is an exhaustion of $\mathbb{R}^n$.
    \begin{claim*}
        For each $i$, there is a diffeomorphism $\Phi_i:\mathbb{R}^n\to\mathbb{R}^n$, such that $\Phi_i(\mathcal{E}_{t_{i+1}}(B_{t_i}(R_i)))=\tilde{B}(R_i)=\mathcal{E}_{t_i}(B_{t_i}(R_i))$ and $\Phi_i$ is identity on $\displaystyle\mathbb{R}^n\setminus\tilde{B}(R_{i+1}-\frac{\varepsilon_{i}}{2})$.
    \end{claim*}
    \begin{proof}[Proof of claim]
        For each $t\in[t_i,t_{i+1}]$, by (\ref{ballnested}), we know that $\mathcal{E}_t$ embeds $B_{t_i}(R_i)$ into $\mathbb{R}^n$, and $\mathcal{E}_t(B_{t_i}(R_i))\subset\tilde{B}(R_{i+1}-\varepsilon_i)$.
        \par
        For an interior point $x$ in the domain of $\mathcal{E}_t$, $t\in[t_i,t_{i+1}]$, let $V(x,t)$ be the tangent vector of the $t$-parameterized curve $\mathcal{E}_t(x)\subset\mathbb{R}^n$,
        \begin{equation}
            \frac{d}{dt}\mathcal{E}_t(x)=V(\mathcal{E}_t(x),t). \label{evolu}
        \end{equation}
        Now choose a smooth time-independent cutoff function $\varphi$ such that $\varphi=1$ on $\tilde{B}(R_{i+1}-\varepsilon_i)$ and $\varphi=0$ on $\displaystyle\mathbb{R}^n\setminus\tilde{B}(R_{i+1}-\frac{\varepsilon_i}{2})$. Let $\Psi_s$, $s\in[0,1]$ be the family of diffeomorphisms generated by $-(t_{i+1}-t_i)\varphi V(x,t_{i+1}-(t_{i+1}-t_i)s)$ with $\Psi_0=id$, namely
        \[
            \left\{
            \begin{aligned}
                \frac{d}{ds}\Psi_s(y)&= -(t_{i+1}-t_i)\varphi V(\Psi_s(y),t_{i+1}-(t_{i+1}-t_i)s), \\
                \Psi_0(y)&=y.
            \end{aligned}
            \right.
        \]
        For all $x\in B_{t_i}(R_i)$, let $\gamma(s):=\Psi_s(\mathcal{E}_{t_{i+1}}(x))$. Then $\varphi(\gamma(s))=1$ if $s\ge0$ sufficiently small. Denote $\displaystyle\eta(\tau):=\gamma\bigg(\frac{t_{i+1}-\tau}{t_{i+1}-t_i}\bigg)$ and $s=\frac{t_{i+1}-\tau}{t_{i+1}-t_i}$. We have
        \[
            \left\{
            \begin{aligned}
                \frac{d}{d\tau}\eta(\tau)&= \frac{d}{ds}\gamma(s)\cdot(-\frac{1}{t_{i+1}-t_i})=V(\eta(\tau),\tau), \\
                \eta(t_{i+1})&=\gamma(0)=\mathcal{E}_{t_{i+1}}(x).
            \end{aligned}
            \right.
        \]
        This is exactly (\ref{evolu}). By the uniqueness of ODE, $\eta(\tau)=\mathcal{E}_\tau(x)$ for small $s\ge0$. Since $\varphi(\mathcal{E}_t(x))=1$ for all $t\in[t_i,t_{i+1}]$, we have $\eta(\tau)=\mathcal{E}_\tau(x)$ for all $s\in[0,1]$. Hence
        \[
            \Psi_1(\mathcal{E}_{t_{i+1}}(x))=\gamma(1)=\eta(t_i)=\mathcal{E}_{t_i}(x).
        \]
        So we have $\Psi_1\circ\mathcal{E}_{t_{i+1}}=\mathcal{E}_{t_i}\text{ on }B_{t_i}(R_i)$. By the construction of the cutoff function $\varphi$, $\Psi_1$ is obviously identity on $\displaystyle\mathbb{R}^n\setminus\tilde{B}(R_{i+1}-\frac{\varepsilon_{i}}{2})$. Then $\Phi_i=\Psi_1$ is the desired diffeomorphism.
    \end{proof}
    Define $\Theta_i:B_{t_i}(R_i)\to\mathbb{R}^n$ by
    \[
        \Theta_i:=\Phi_1\cdots\Phi_{i-1}\mathcal{E}_{t_i}.
    \]
    By the definition of $\Phi_i$, $\Theta_i$ agrees with $\Theta_j$ on $B_{t_j}(R_j)$ for each $j<i$. Then we can let $i\to+\infty$ to obtain a diffeomorphism from $M$ to $\mathbb{R}^n$.
\end{proof}

\begin{proof}[Proof of Theorem \ref{mainthm1} (Method 2)]
    The proof is just a slight modification from the construction in  \cite[Theorem 5.7]{Wang20}. We provide details here for completeness. 
    
    Fix $x_0\in M$. Define
    \[
        M^\prime:=\{(x,t)\big|d_{g(t)}(x,x_0)=\frac{\beta}{2}\sqrt{t}\}.
    \]
    By the injectivity condition (\romannumeral2), $M^\prime=\bigcup_{t\ge0}\partial B_{g(t)}(x_0,\frac{\beta}{2}\sqrt{t})$ is a smooth manifold. Define a space-time exponential map as follows
    \begin{align*}
        Exp:\quad\mathbb{R}^n\times[0,+\infty)&\to M\times[0,+\infty), \\
        (v,t)&\mapsto(exp_{x_0,g(t)}v,t).
    \end{align*}
    Let $\Omega:=\{(x,t)\big|d_{g(t)}(x,x_0)\le\frac{3\beta}{4}\sqrt{t}\}$. Then by the injectivity condition (\romannumeral2) again, $Exp|_\Omega$ is a diffeomorphism. Since $M^\prime\subset\Omega$, we know that $M^\prime$ is diffeomorphic to $Exp^{-1}(M^\prime)=\{(v,t)\big||v|=\frac{\beta}{2}\sqrt{t}\}$ via the diffeomorphism map $Exp^{-1}$. Note that $\{(v,t)\big||v|=\frac{\beta}{2}\sqrt{t}\}$ is diffeomorphic to $\mathbb{R}^n$ through the projection $P^\prime$ defined as
    \[
        P^\prime(v,t):=v\in\mathbb{R}^n.
    \]
    So we obtained the diffeomorphism from $M^\prime$ to $\mathbb{R}^n$: $P^\prime\circ Exp^{-1}$.

    To show $M$ is diffeomorphic to $M^\prime$, we define the projection map $P$ from $M^\prime$ to $M$ by
    \[
        P(x,t):=x,\quad\forall(x,t)\in M^\prime.
    \]
    Since the projection map $P$ is smooth and non-degenerate everywhere, it suffices to show that $P$ is both injective and surjective. 
    
    We prove the injectivity first. Suppose $P(x,t_1)=P(y,t_2)$ for some $x,y\in M$ and $t_1\le t_2$. It is clear that $x=y$. If $t_1=0$, then $x=y=x_0$ and hence $t_2=0$. For $0<t_1\le t_2$, we have
    \begin{equation}
        \frac{d_{g(t_1)}(x,x_0)}{\sqrt{t_1}}=\frac{d_{g(t_2)}(x,x_0)}{\sqrt{t_2}}=\frac{\beta}{2}. \label{injeq}
    \end{equation}
    Combining with (\ref{distdist}), we arrive at
    \[
        \bigg(\frac{t_2}{t_1}\bigg)^\frac{1}{2}=\frac{d_{g(t_2)}(x,x_0)}{d_{g(t_1)}(x,x_0)}\le\bigg(\frac{t_2}{t_1}\bigg)^\psi,
    \]
    which implies $t_2\le t_1$ since $\psi<\frac{1}{2}$. Then we have $t_1=t_2$ and the injectivity is obtained.

    We now show that $P$ is surjective. It is trivial that $P(x_0,0)=x_0$. We need to prove that for each $x\in M\setminus\{x_0\}$, there exists $t>0$ such that
    \[
        d_{g(t)}(x,x_0)=\frac{\beta}{2}\sqrt{t}.
    \]
    This is true by noting that the function $\frac{d_{g(t)}(x,x_0)}{\sqrt{t}}$ is continuous in $t$ and the fact that 
    \[
        \lim_{t\to0}\frac{d_{g(t)}(x,x_0)}{\sqrt{t}}=+\infty
    \]
    and
    \[
        \lim_{t\to+\infty}\frac{d_{g(t)}(x,x_0)}{\sqrt{t}}\le\lim_{t\to+\infty}\frac{d_{g(1)}(x,x_0)\cdot t^\psi}{\sqrt{t}}=0,
    \]
    where we used (\ref{distdist}).
    
    Therefore, we prove that $P: M^\prime\to M$ is a diffeomorphism. Combining with the one we have already obtained from $M^\prime$ to $\mathbb{R}^n$, we obtain the desired diffeomorphism map from $M$ to $\mathbb{R}^n$: $P^\prime\circ Exp^{-1}\circ P^{-1}$. This completes the proof.
\end{proof}

\section{Applications}

\begin{proof}[Proof of Corollary \ref{small curvature concentration}]
    By the proof of \cite[Theorem 1.3]{CHL24} and \cite[Theorem 1.4]{Mar25}, there exists a long-time solution $g(t)$ to the Ricci flow with $g(0)=g_0$ and
    \[
        \left\{
            \begin{aligned}
                &|\Rm(g(t))|\le C_0\sigma t^{-1}; \\
                &\Inj(g(t))\ge C_0^{-1}\sqrt{t},
            \end{aligned}
        \right.
    \]
    where $C_0=C_0(n,A,v_0)>0$. Since $\sigma$ can be small, the result follows from Theorem \ref{mainthm1}. 
\end{proof}

\begin{proof}[Proof of Corollary \ref{3d}]
    By assumptions, there exists $v_0>0$ such that for all $x\in M$ and $r>0$, we have
    \begin{equation}
        \Vol_{g_0}(B_{g_0}(x,r))\ge v_0r^n. \label{mvg}
    \end{equation}
    By \cite[Theorem 1.7]{ST21}, there exists $T=T(v_0)>0$ and a smooth complete Ricci flow $\tilde{g}(t)$ defined on $M\times[0,T]$ with $\tilde{g}(0)=g_0$. It follows from \cite{CXZ13} that non-negativity of Ricci curvature is preserved. By \cite[Lemma 4.1]{ST21}, there exist $\beta=\beta(v_0)>0$, $C_0=C_0(v_0)>0$ and $\hat{T}=\hat{T}(v_0)\in(0,T]$ such that
    \begin{equation} \label{curinjbound}
        \left\{
        \begin{aligned}
            &|\Rm(\tilde{g}(t))|\le \frac{C_0}{t}; \\
            &\Inj(\tilde{g}(t))\ge\beta\sqrt{t}
        \end{aligned}
        \right.
    \end{equation}
    for all $t\in[0,\hat{T}]$.

    Since the assumptions and (\ref{curinjbound}) are scaling invariant, by applying the argument above to $R^{-2}g_0$ for $R\to+\infty$ and re-scaling it back, we obtain a sequence of Ricci flow $g_R(t)$ on $[0,\hat{T}R^2]$ with $g_R(0)=g_0$ for any $R>0$. Using the argument in the proof of \cite[Theorem 1.1]{CHL24}, we can extract convergent sub-sequence in locally smooth sense to obtain a smooth complete long-time solution $g(t)$ to the Ricci flow with $g(0)=g_0$ and
    \[
        \left\{
        \begin{aligned}
            &\Ric(g(t))\ge0; \\
            &\Inj(g(t))\ge\beta\sqrt{t}.
        \end{aligned}
        \right.
    \]
    Then the result follows from Theorem \ref{mainthm1} or the proof of \cite[Theorem 1.1]{HL21}.
\end{proof}

\begin{proof}[Proof of Corollary \ref{bounded curvature}]
    By \cite[Theorem 1.2]{CM24}, there exists a smooth complete long-time solution $g(t)$ to the Ricci flow with $g(0)=g_0$ and
    \[
        \left\{
            \begin{aligned}
                &|\Rm(g(t))|\le (3e)^\frac{n}{2}t^{-1}; \\
                &\lim_{t\to+\infty}t\sup_M|\Rm(g(t))|=0.
            \end{aligned}
        \right.
    \]
    By \cite[Corollary 4.3]{CM24} and \cite[Lemma 3.1]{Mar25Sg}, there exists $v_0=v_0(n,C_{g_0})>0$ such that
    \[
    \Vol_{g(t)}B_{g(t)}(x,\sqrt{t})\ge v_0t^\frac{n}{2}\text{ for all }t>0.
    \]
    Let $\tilde{g}:=\frac{1}{t}g(t)$. Then we have
    \[
        \left\{
            \begin{aligned}
                &|\Rm(\tilde{g})|\le(3e)^\frac{n}{2}; \\
                &\Vol_{\tilde{g}}B_{\tilde{g}}(x,1)\ge v_0.
            \end{aligned}
        \right.
    \]
    By \cite[Theorem 4.7]{CGT82}, there exists $\beta=\beta(n,C_{g_0})>0$ such that $\Inj(\tilde{g})\ge\beta$, which implies $\Inj(g(t))\ge\beta\sqrt{t}$.
   
    Since $\,\lim_{t\to+\infty}t\sup_M|\Rm(g(t))|=0$, then the curvature condition in Theorem \ref{mainthm1} is satisfied for any $\psi>0$. Indeed, since there exists $T_0>0$ such that $|\Rm(g(t)|\le\frac{\psi}{(n-1)t}$ for all $t\ge T_0$, it suffices to consider $\tilde{g}(t)=g(T_0+t)$. By Theorem \ref{mainthm1}, the result follows. 
\end{proof}

\bibliographystyle{amsplain}
\bibliography{refs}

\end{document}